\documentclass[12pt]{amsart}
\usepackage{amssymb}
\usepackage{mathtools}
\usepackage[english]{babel}
\usepackage{epsfig}
\usepackage{mathptmx}
\usepackage{amsmath,amsfonts,amssymb}
\usepackage{mathtools}
\usepackage{mathrsfs}
\usepackage[all]{xy}
\usepackage{graphicx}
\usepackage{latexsym}
\usepackage{verbatim}

\numberwithin{equation}{section}

\def\Re{{\sf Re}\,}
\def\Im{{\sf Im}\,}

\newcommand{\D}{\mathbb D}

\newcommand{\de}{\partial}
\newcommand{\R}{\mathbb R}
\newcommand{\Ha}{\mathbb H}

\newcommand{\C}{\mathbb C}

\newcommand{\N}{\mathbb N}

\def\Re{{\sf Re}\,}
\def\Im{{\sf Im}\,}

\newcommand{\strip}{\mathbb{S}}

\newcommand{\UD}{\mathbb{D}}

\def\Re{{\sf Re}\,}
\def\Im{{\sf Im}\,}

\def\Re{{\sf Re}\,}
\def\Im{{\sf Im}\,}

\def\1#1{\overline{#1}}
\def\2#1{\widetilde{#1}}
\def\3#1{\widehat{#1}}
\def\4#1{\mathbb{#1}}
\def\5#1{\frak{#1}}
\def\6#1{{\mathcal{#1}}}

\def\Re{{\sf Re}\,}
\def\Im{{\sf Im}\,}

\newcommand{\mcite}[1]{\csname b@#1\endcsname}

\theoremstyle{theorem}

\def\Re{{\sf Re}\,}
\def\Im{{\sf Im}\,}

\newtheorem{theorem}{Theorem}[section]
\newtheorem{lemma}[theorem]{Lemma}
\newtheorem{proposition}[theorem]{Proposition}
\newtheorem{corollary}[theorem]{Corollary}

\theoremstyle{definition}
\newtheorem{definition}[theorem]{Definition}

\theoremstyle{remark}
\newtheorem{remark}[theorem]{Remark}

\numberwithin{equation}{section}

\begin{document}
\title[On the K\"onigs function]{On the K\"onigs function of semigroups of holomorphic self-maps of the unit disc}
\date{\today}
\thanks{$^\dag$ Partially supported by the \textit{Ministerio
de Econom\'{\i}a y Competitividad} and the European Union (FEDER), project MTM2015-63699-P, and  by \textit{La Consejer\'{\i}a de Econom\'{\i}a y Competitividad de la Junta de Andaluc\'{\i}a}.}

\author[F. Bracci]{Filippo Bracci}
\address{F. Bracci: Dipartimento di Matematica, Universit\`a di Roma ``Tor Vergata", Via della Ricerca
Scientifica 1, 00133, Roma, Italia.} \email{fbracci@mat.uniroma2.it}

\author[M. D. Contreras]{Manuel D. Contreras$^\dag$}

\author[S. D\'{\i}az-Madrigal]{Santiago D\'{\i}az-Madrigal$^\dag$}
\address{M. D. Contreras, S. D\'{\i}az-Madrigal: Camino de los Descubrimientos, s/n\\
Departamento de Matem\'{a}tica Aplicada~II and IMUS\\ Universidad de Sevilla\\ Sevilla,
41092\\ Spain.}\email{contreras@us.es} \email{madrigal@us.es}

\dedicatory{In Memory of our beloved friend Sasha Vasil'ev}

\begin{abstract} Let $(\phi_t)$ be a semigroup of holomorphic self-maps of~$\UD$. In this note, we use an abstract approach to define the K\"onigs function of $(\phi_t)$ and ``holomorphic models'' and show how to deduce the existence and properties of the infinitesimal generator of $(\phi_t)$ from this construction.
\end{abstract}
\maketitle

\section{Introduction}
Let $\Omega$ be a Riemann surface. A (one-parameter) semigroup $(\phi_t)$ of holomorphic self-maps of~$\Omega$ is a continuous homomorphism $t\mapsto \phi_t$ from the
additive semigroup $(\R_{\ge0}, +)$ of non-negative real numbers to the
semigroup $({\sf Hol}(\Omega,\Omega),\circ)$ of  holomorphic self-maps
of $\Omega$ with respect to composition, endowed with the
topology of uniform convergence on compacta. If $\phi_t(\Omega)=\Omega$ for some---and hence for all---$t>0$, every $\phi_t$ is an automorphism of $\Omega$ and $(\phi_t)$ will be called a group. 

Let $(\phi_t)$ be a semigroup in $\D:=\{\zeta\in \C: |\zeta|<1\}$, the unit disc of $\C$. It is well known that  $\phi_{t_0}$ has a fixed point in $\D$ for some $t_0>0$ if and only if there exists $x\in \D$ such that $\phi_t(x)=x$  for all $t\geq 0$. In such a case, the semigroup $(\phi_t)$ is called {\sl elliptic}. Moreover, there exists $\lambda\in \C$ with $\Re \lambda\geq 0$ such that $\phi_t'(x)=e^{-\lambda t}$ for all $t\geq 0$. This number $\lambda$ is called the {\sl spectral value} of the elliptic semigroup. If $(\phi_t)$ is not the trivial (semi)group $\phi_t(z)\equiv z$ for all $t\geq 0$, then $\lambda\neq0$ and $x$ is the unique fixed point of $(\phi_t)$ and, in this case, $x$ is called the {\sl Denjoy-Wolff point} of the semigroup. If $(\phi_t)$ is a non-trivial elliptic semigroup with spectral value $\lambda$, then $(\phi_t)$ is a group if and only if $\Re \lambda=0$.

If the semigroup $(\phi_t)$ is not elliptic, then there exists a unique $x\in \de \D$ such that there is $\lambda\geq 0$ with
\[
\liminf_{z\to x} \frac{1-|\phi_t(z)|}{1-|z|}=e^{-\lambda t},\quad t\geq 0.
\]
The point $x$ is called the Denjoy-Wolff point of the semigroup and $\lambda$  the {\sl spectral value} of the non-elliptic semigroup.
Moreover, this semigroup is said {\sl hyperbolic} if the spectral value is non-zero, while it is said {\sl parabolic} if the spectral value is zero.

From a dynamical point of view, the basic result about semigroups is the so called continuous Denjoy-Wolff theorem. This result says that if the semigroup is  elliptic with $\Re\lambda >0$ or it is non-elliptic and, in both cases, $\tau\in\overline{\D}$ is the Denjoy-Wolff point of the semigroup, it holds
$$
\lim_{t\to+\infty}\phi_t(z)=\tau,
$$
for all $z\in\D$.

The two main objects associated to a semigroup are the infinitesimal generator and the K\"onigs function. The {\sl infinitesimal generator} $G$ of a semigroup $(\phi_t)$ in $\D$ is a holomorphic function $G:\D\to \C$ such that
\[
\frac{\de \phi_t(z)}{\de
t}=G(\phi_t(z))\qquad\text{for all $z\in\UD$ and all~$t\ge0$}.
\]
A very famous result due to Berkson and Porta \cite{BerPor78}  characterizes those holomorphic functions which are infinitesimal generators of a semigroup in $\D$:

\begin{theorem} A holomorphic function $G:\D\to \C$ is the infinitesimal generator of a semigroup if and only if there exist a point
$\tau\in\overline{\D}$ and a holomorphic function $p:\D\to \C$  with $\Re
p(z)\geq 0$ for all $z\in \D$ such that the following formula holds
\begin{equation}\label{Eq:Berkson-Porta formula}
G(z)=(z-\tau)(\overline{\tau}z-1)p(z),\quad z\in\D.
\end{equation}
\end{theorem}

On the other hand, the K\"onigs function of a semigroup $(\phi_t)$ is a univalent function on $\D$ which ``simultaneously linearizes'' the $\phi_t$'s. More precisely, if the semigroup is elliptic with spectral value $\lambda\neq0$, there exists an essentially unique univalent function $h:\D\to\C$ such that $h\circ\phi_t=e^{-\lambda t}h$, for $t\geq 0$. Likewise, if the semigroup is non-elliptic, there exists an essentially unique univalente function $h:\D\to\C$ such that $h\circ\phi_t=h+it$, for $t\geq 0$.

The usual way (see for instance \cite{Ababook89}, \cite{Sis85}) to define these two elements for a general semigroup is to generate firstly the infinitesimal generator solving the associated Cauchy problem and defining then the K\"onigs function using properties of the infinitesimal generator. This point of view, which is very useful to get analytic information on the K\"onigs function, has however a main drawback: it is not clear whether all ``linear models'' one can obtain are essentially the same or can be factorize through the K\"onigs function. 

The aim of this note is to define first the K\"onigs function of a semigroup by using a categorial construction which was first partially introduced in \cite{Co} and completely developed in \cite{AroBra16}, and then to show how to recover the existence of infinitesimal generators and the Berkson-Porta formula from this. This point of view allows to obtain immediately the uniqueness and functorial properties for every linearization model.  

For the sake of clearness, we are going to restrict ourselves to the (more delicate) non-elliptic case. Nevertheless, our arguments can be easily adapted to the elliptic case.
As one might expect this approach leads to  new proofs of several well-known results. At this respect, we want to underline our proof of Theorem \ref{Thm:starlike-infinity} and the deduction of the Berkson and Porta's decomposition theorem.

\section{The divergence rate and hyperbolic steps}

Given a Riemann surface  $\Omega$, we denote by $k_\Omega (z,w)$, $z,w\in \Omega$, the {\sl hyperbolic distance} of $z$ and $w$ in $\Omega$. We simply write $\omega:=k_\D$.

We recall here some results from \cite[Section 2]{AroBra16}. For the sake of completeness, we sketch some proofs.
Our first step is to introduce a quantity, called {\sl divergence rate}, which, roughly speaking, measures the average hyperbolic speed of escape of an orbit of a semigroup.
\begin{lemma}\label{Lem:def-hyper-step}
Let $(\phi_t)$ be a continuous semigroup on a Riemann surface $\Omega$. Then for all $z\in \Omega$ the limit
\begin{equation}\label{Eq:def-speed-of-divergence}
c_\Omega(\phi_t):=\lim_{s\to+\infty}\frac{k_\Omega(\phi_s(z), z)}{s}
\end{equation}
exists independently of $z$, $c_\Omega(\phi_t)\in [0,+\infty)$ and moreover
\begin{equation}\label{Eq:speed-of-divergence}
c_\Omega(\phi_t)=\inf_{s> 0}\frac{k_\Omega(\phi_s(z), z)}{s}.
\end{equation}
This number $c_\Omega(\phi_t)$ is called the {\sl divergence rate of $(\phi_t)$}.
\end{lemma}

\begin{proof}
Fix $z\in \Omega$. For $s,t\geq 0$, using the triangle inequality and the contractiveness property of $k_\Omega$ under holomorphic maps, it holds
\begin{equation*}
\begin{split}
k_\Omega(z,\phi_{t+s}(z))&\leq k_\Omega(z,\phi_t(z))+ k_\Omega(\phi_t(z),\phi_{t+s}(z))\\& = k_\Omega(z,\phi_t(z))+ k_\Omega(\phi_t(z),\phi_t(\phi_s((z))))\leq k_\Omega(z,\phi_t(z))+  k_\Omega(z,\phi_s(z)).
\end{split}
\end{equation*}
Hence, the function $[0,+\infty)\ni t\mapsto  k_\Omega(z,\phi_t(z))$ is a non-negative continuous subadditive function. By Fekete's Theorem, the limit
\[
c(z):=\lim_{t\to+\infty}\frac{k_\Omega(\phi_t(z), z)}{t}=\inf_{t> 0}\frac{k_\Omega(\phi_t(z), z)}{t}
\]
exists finitely.

It remains to show that $c(z)$ is independent of $z\in \Omega$.  To this aim, let $w\in \Omega$ be another point. Then, using again the triangle inequality and the contractiveness of the hyperbolic distance, we have
\begin{equation*}
\begin{split}
k_\Omega(z,\phi_t(z))&\leq  k_\Omega(z,w)+k_\Omega(w,\phi_t(w))+k_\Omega(\phi_t(z), \phi_t(w))\\&\leq k_\Omega(w,\phi_t(w))+2k_\Omega(z,w).
\end{split}
\end{equation*}
Therefore, dividing by $t$ and taking the limit as $t\to \infty$, we see that $c(z)\leq c(w)$. Changing $z$ with $w$ and repeating the previous argument we get $c(z)=c(w)$.
\end{proof}
\begin{remark}\label{Rem:speed-elliptic}
It is clear from the definition that if $(\phi_t)$ is a continuous semigroup in $\Omega$ such that there exists $z\in\Omega$ with $\phi_t(z)=z$ for all $t\geq 0$, then the divergence rate of $(\phi_t)$ is $c_\Omega(\phi_t)=0$. In particular, elliptic semigroups in $\D$ have zero divergence rate.
\end{remark}

We introduce now another measure of ``hyperbolicity'' of a semigroup on a Riemann surface $\Omega$, and relate it to the divergence rate.

Let $(\phi_t)$ be a continuous semigroup on a Riemann surface $\Omega$. Note that for $r\geq r'\geq 0$,
\begin{equation*}
\begin{split}
k_\Omega(\phi_r(z), \phi_{r+u}(z))&=k_\Omega(\phi_r(z), \phi_{r}(\phi_u(z)))=
k_\Omega(\phi_{r-r'}(\phi_{r'}(z)), \phi_{r-r'}(\phi_{r'}(\phi_u(z))))\\&
\leq k_\Omega(\phi_{r'}(z), \phi_{r'}(\phi_u(z)))=k_\Omega(\phi_{r'}(z), \phi_{r'+u}(z)).
\end{split}
\end{equation*}
Hence, the function $r\mapsto k_\Omega(\phi_r(z), \phi_{r+u}(z))$ is decreasing in $r$ and therefore the limit exists. Taking this into account, we can define the hyperbolic step of a semigroup:

\begin{definition}
Let $(\phi_t)$ be a continuous semigroup on a Riemann surface $\Omega$. Let $u\geq 0$. The {\sl hyperbolic  step of order $u$ } (or {\sl $u$-th hyperbolic step}) of $(\phi_t)$ at $z\in \Omega$ \index{Hyperbolic steps of order $u$} is defined as
\[
s_u(\phi_t,z):=\lim_{r\to +\infty}k_\Omega(\phi_r(z), \phi_{r+u}(z)).
\]
The $1$-st hyperbolic step $s_1(\phi_t, z)$ is just called  {\sl hyperbolic step}.
\end{definition}

\begin{remark}\label{Rem:step-elliptic}
If $(\phi_t)$ is a continuous group of automorphisms of a Riemann surface $\Omega$, it holds $s_u(\phi_t, z)=k_\Omega(z,\phi_u(z))$.

On the other hand, if $(\phi_t)$ is an elliptic semigroup, not a group, in $\D$ then $s_u(\phi_t,z)\equiv 0$ for all $z\in \D$ and $u\geq 0$. Indeed, if $\tau\in \D$ is the Denjoy-Wolff point of $(\phi_t)$ then $\phi_t(\tau)=\tau$ for all $t\geq 0$ and hence $s_u(\phi_t,\tau)= 0$. If $z\in \D$ is not the Denjoy-Wolff point of $(\phi_t)$, then by the continuous Denjoy-Wolff theorem, it holds $\lim_{t\to +\infty}\omega(\phi_t(z), \phi_{t+u}(z))=\omega(\tau,\tau)=0$, hence $s_u(\phi_t,z)= 0$.
\end{remark}

\begin{proposition}\label{Prop:higher-step-speed}
Let $(\phi_t)$ be a continuous semigroup on a Riemann surface $\Omega$. Then for all $z\in \Omega$ it holds
\[
c_\Omega(\phi_t)=\lim_{u\to +\infty}\frac{s_u(\phi_t,z)}{u}.
\]
\end{proposition}
\begin{proof}
Clearly $k_\Omega(z, \phi_u(z))\geq s_u(\phi_t,z)$ for all $u\geq 0$ because of the contractiveness of the hyperbolic distance with respect to holomorphic functions. Then,
\[
c_\Omega(\phi_t)\geq \limsup_{u\to +\infty}\frac{s_u(\phi_t,z)}{u}.
\]

In order to prove the converse, let $m\in \N$ and $u> 0$. By the triangle inequality,
\[
\frac{k_\Omega(z, \phi_{um}(z))}{um}\leq \frac{1}{um}\sum_{j=0}^{m-1}k_\Omega(\phi_{uj}(z), \phi_{u(j+1)}(z)).
\]
Note that $k_\Omega(\phi_{uj}(z), \phi_{u(j+1)}(z))\to s_u(\phi_t, z)$ as $j\to \infty$, hence, by the Ces\`aro Means Theorem  the right-hand side of the previous inequality converges to $\frac{s_u(\phi_t,z)}{u}$ as $m\to \infty$. Therefore, for any $u>0$ it holds
\[
\lim_{t\to +\infty}\frac{k_\Omega(z,\phi_t(z))}{t}=
\lim_{\N\ni m\to \infty}\frac{k_\Omega(z,\phi_{mu}(z))}{um}\leq \frac{s_u(\phi_t,z)}{u}.
\]
This proves that
\[
c_\Omega(\phi_t)\leq \liminf_{u\to +\infty}\frac{s_u(\phi_t,z)}{u},
\]
and the result follows.
\end{proof}

In case of a non-elliptic semigroup in $\D$ the divergence rate is essentially equal to the spectral value.
\begin{theorem}\label{Thm:speed-dilation}
Let $(\phi_t)$ be a non-elliptic semigroup in $\D$ with Denjoy-Wolff point $\tau\in \partial\D$ and spectral value $\lambda\geq 0$. Let $c_\D(\phi_t)$ denote the divergence rate of $(\phi_t)$. Then
\begin{equation}\label{Eq:speed-divergence}
c_\D(\phi_t)=\frac{1}{2} \lambda.
\end{equation}
In particular, if $(\phi_t)$ is hyperbolic then $c_\D(\phi_t)>0$ while if $(\phi_t)$ is parabolic then $c_\D(\phi_t)=0$.
\end{theorem}

\begin{proof}
Set $c:=c_\D(\phi_t)$. Let $z_m:=\phi_m(0)$, $m\in \N$. By the continuous Denjoy-Wolff theorem,  $\lim_{m\to \infty}z_m=\tau$. Moreover
\[
\liminf_{m\to \infty}\frac{1-|z_{m+1}|}{1-|z_m|}=\liminf_{m\to \infty}\frac{1-|\phi_1(z_m)|}{1-|z_m|}\geq e^{-\lambda}.
\]
Hence
\[
\liminf_{m\to \infty}(1-|z_m|)^{1/m}\geq \liminf_{m\to \infty}\frac{1-|\phi_1(z_m)|}{1-|z_m|}\geq e^{-\lambda}.
\]
Thus,
\begin{equation*}
\begin{split}
c&=\lim_{t\to \infty}\frac{\omega(0,\phi_t(0))}{t}=\lim_{\N\ni m\to \infty}\frac{\omega(0,\phi_m(0))}{m}
=\lim_{m\to \infty}\frac{\omega(0,z_m)}{m}\\&=\frac{1}{2}\lim_{m\to \infty}\log\left(\frac{1+|z_m|}{1-|z_m|} \right)^{1/m}\leq -\frac{1}{2}\log e^{-\lambda}=\frac{1}{2}\lambda.
\end{split}
\end{equation*}
If $(\phi_t)$ is parabolic, that is $\lambda=0$, that estimate implies that $c=0$ as well.

In case $(\phi_t)$ is hyperbolic, thus $\lambda>0$, let $R\in (0,1)$ and let $E(\tau, R)$ be a horocycle of center $\tau$ and radius $R$. Since the  point of $\overline{E(\tau, R)}$ closest to $0$ is $\frac{1-R}{1+R}\tau$, it holds
\begin{equation}\label{Eq:inf-horosph}
\inf_{z\in E(\tau, R)}\omega (0,z)=\omega(0,\frac{1-R}{1+R}\tau)=-\frac{1}{2}\log R.
\end{equation}
Since $0\in \partial E(\tau, 1)$, by Julia's Lemma it follows that $z_m=\phi_m(0)=\phi_1^{\circ m}(0)\in \overline{E(\tau, e^{-m\lambda})}$ for all $m\in \N$. Hence, by \eqref {Eq:inf-horosph},
\[
\frac{1}{m}\omega(0, z_m)\geq \frac{1}{m}\omega(0, \frac{1-e^{-m\lambda}}{1+e^{-m\lambda}}\tau)=-\frac{1}{2m}\log e^{-m\lambda}=\frac{1}{2}\lambda.
\]
Therefore,
\[
c=\lim_{m\to \infty}\frac{1}{m}\omega(0, z_m)\geq \frac{1}{2}\lambda.
\]
\end{proof}

This result suggest the following very useful definition.
\begin{definition} Let $(\phi_t)$ be a continuous semigroup on a Riemann surface $\Omega$ and suppose that $\phi_t$ has no common fixed point, that is, there is no $z\in\Omega$ such that $\phi_t(z)=z$, for all $t>0$. The semigroup is said hyperbolic if $c_\Omega(\phi_t)>0$ while it is said parabolic if $c_\Omega(\phi_t)=0$.
\end{definition}

Formula \eqref{Eq:speed-divergence} allows easily to see how the spectral value of a semigroup behaves under conjugacy. To set up properly the terminology, we give the following definition:

\begin{definition}
Let $(\phi_t)$  be a semigroup in $\D$. Let $\Omega$ be a Riemann surface  and let $(\varphi_t)$ be a semigroup in $\Omega$. We say that $(\phi_t)$ is  {\sl semi-conjugated to}\index{Semigroup!semi-conjugated} $(\varphi_t)$ if there exists a holomorphic map  $g: \D \to \Omega$, called  a {\sl semi-conjugation map}\index{Semi-conjugation map}, such that  for all $t\geq 0$ it holds $g\circ \phi_t=\varphi_t\circ g$.
\end{definition}

\begin{proposition}\label{Prop:speed-of-div-semiconjugacy}
Let $(\phi_t)$  be a semigroup in $\D$. Suppose that $(\phi_t)$ is semi-conjugated to a continuous semigroup $(\varphi_t)$ of a Riemann surface $\Omega$ via the semi-conjugation map $g$. Then
\begin{equation}\label{Eq:paso-abajo}
c_\D(\phi_t)\geq c_\Omega(\varphi_t).
\end{equation}
\end{proposition}

\begin{proof}
Since $g:\D \to \Omega$ is holomorphic, it contracts the hyperbolic distance. Hence,
\[
k_\D(0,\phi_t(0))\geq k_\Omega(g(0), g(\phi_t(0)))=k_\Omega(g(0), \varphi_t(g(0)).
\]
Dividing by $t$ and taking the limit as $t\to \infty$, formula \eqref{Eq:paso-abajo} follows at once.
\end{proof}

\begin{corollary}
Let $(\phi_t), (\varphi_t)$  be two semigroups in $\D$ and suppose that $(\phi_t)$ is semi-conjugated to   $(\varphi_t)$. If $(\varphi_t)$ is hyperbolic with spectral value $\eta>0$, then $(\phi_t)$ is hyperbolic and its spectral value $\lambda$ satisfies
$\lambda\geq \eta$.
\end{corollary}

\begin{proof}
Let $c$ denote the divergence rate of $(\phi_t)$ and $\tilde{c}$ that of $(\varphi_t)$. By \eqref{Eq:speed-divergence}, $\tilde{c}=\frac{1}{2}\eta>0$. Hence, by \eqref{Eq:paso-abajo}, $c\geq \tilde{c}>0$, which implies in particular that $(\phi_t)$ is non-elliptic. Therefore, again by \eqref{Eq:speed-divergence},
\[
\lambda=2c\geq 2\tilde{c}=\eta>0,
\]
and $(\phi_t)$ is hyperbolic.
\end{proof}

\section{Holomorphic models}

\begin{definition}\label{Def:semi-model}
 Let $(\phi_t)$ be a  semigroup in $\D$. A {\sl holomorphic model}  for $(\phi_t)$ is a triple  $(\Omega,h,\psi_t)$  where
$\Omega$ is a Riemann surface, $h: \D \to \Omega$ is univalent and $(\psi_t)$ is a continuous group of automorphisms of $\Omega$ such that
\begin{equation}\label{Eq:unosemigroup}
h\circ \phi_t=\psi_t\circ h,\quad t\geq 0,
\end{equation}
and
\begin{equation}\label{Eq:duesemigroup}
\bigcup_{t\geq 0} \psi_{-t}(h(\D))=\Omega.
\end{equation}

We call the manifold $\Omega$ the {\sl base space} \index{Base space of a holomorphic semi-model} and the mapping $h$ the {\sl intertwining mapping}. \index{Intertwining mapping}
\end{definition}
Notice that given a model $(\Omega, h, \Phi_t)$ for a semigroup $(\phi_t)$ of holomorphic self-maps of $\D$ then $(\phi_t)$ is a group if and only if $h(\D)=\Omega$.

The previous notion of holomorphic model was introduced in \cite{AroBra16}, where it was proved that every semigroup of holomorphic self-maps of any complex manifold admits a  holomorphic model, unique up to ``holomorphic equivalence''. The central  idea of this result is to define an abstract space (the space of orbits of a semigroup or  abstract basin of attraction) which inherits a complex structure of a simply connected Riemann surface, in such a way that the semigroup  is conjugated to a continuous group of automorphisms of such a Riemann surface. Moreover, a model is ``universal'' in the sense that every other conjugation of the semigroup to a group of automorphisms factorizes through the model (see \cite[Section 6]{AroBra16} for more details).

\begin{theorem}\label{Thm:existmodel}
Let $(\phi_t)$ be a semigroup in $\D$. Then there exists a holomorphic model $(\Omega, h, \psi_t)$ for $(\phi_t)$. Moreover, either $\Omega$ is biholomorphic to $\D$ or $\Omega$ is biholomorphic to $\C$.
\end{theorem}

The model respects basic properties of the semigroup, in particular, the divergence rate.
\begin{lemma}\label{Lem:speed-and-model}
Let $(\phi_t)$ be a semigroup in $\D$. Let $(\Omega, h, \psi_t)$ be a holomorphic model for $(\phi_t)$. Then
\begin{equation}\label{Eq:distance-limit}
k_\Omega(h(z),h(w))=\lim_{t\to +\infty}\omega(\phi_t(z), \phi_t(w)) \quad z,w\in \D.
\end{equation}
Moreover, $c_\D(\phi_t)=c_\Omega(\psi_t)$.
\end{lemma}

\begin{proof}
By \eqref{Eq:unosemigroup}, for $0\leq s\leq t$, we have
\[
\psi_{-s}(h(\D))=\psi_{-t}(\psi_{t-s}(h(\D)))=\psi_{-t}(h(\phi_{t-s}(\D)))\subset \psi_{-t}(h(\D)),
\]
hence the Riemann surface $\Omega$ is the growing union of  $\{\psi_{-t}(h(\D))\}_{t\geq 0}$. Taking into account that $h$ is an isometry between $\omega$ and $k_{h(\D)}$, we have
\begin{equation*}
\begin{split}
k_\Omega(h(z), h(w))&=\lim_{t\to +\infty}k_{\psi_{-t}(h(\D))}(h(z), h(w))=\lim_{t\to +\infty} k_{h(\D)}(\psi_t(h(z)), \psi_t(h(w)))\\ &= \lim_{t\to +\infty} k_{h(\D)}(h(\phi_t(z)), h(\phi_t(w)))=\lim_{t\to +\infty} \omega(\phi_t(z), \phi_t(w)),
\end{split}
\end{equation*}
which proves \eqref{Eq:distance-limit}.

Finally, let $z\in \D$. Then for all $u\geq 0$,  the $u$-th hyperbolic step of $(\phi_t)$ satisfies
\begin{equation*}
\begin{split}
s_u(\phi_t, z)&=\lim_{v\to +\infty}\omega(\phi_v(z), \phi_{v+u}(z))=\lim_{v\to +\infty}\omega(\phi_v(z), \phi_{v}(\phi_u(z)))\\&\stackrel{\eqref{Eq:distance-limit}}{=} k_\Omega(h(z), h(\phi_u(z)))=k_\Omega(h(z), \psi_u(h(z)))\\&=k_\Omega(\psi_v(h(z)), \psi_{v+u}(h(z)))=\lim_{v\to +\infty}k_\Omega(\psi_v(h(z)), \psi_{v+u}(h(z)))\\&=s_u(\psi_t, h(z)),
\end{split}
\end{equation*}
where we used that $\psi_v$ is an isometry for $k_\Omega$ for all $v\geq 0$. Hence, $c_\D(\phi_t)=c_\Omega(\psi_t)$ by Proposition \ref{Prop:higher-step-speed}.
\end{proof}

A first consequence of the previous result and the results of the previous section is that holomorphic models detect the type of the semigroups:

\begin{corollary}\label{Prop:zero-pos-semi-b}
Let $(\phi_t)$ be a non-elliptic semigroup in $\D$ and let $(\Omega, h, \psi_t)$ be a holomorphic model for $(\phi_t)$. Then
\begin{enumerate}
\item $(\phi_t)$ is hyperbolic if and only if $c_\Omega(\psi_t)>0$,
\item $(\phi_t)$ is parabolic if and only if $c_\Omega(\psi_t)=0$.
\end{enumerate}
\end{corollary}

\begin{proof}
The result follows directly from Theorem \ref{Thm:speed-dilation} and Lemma \ref{Lem:speed-and-model}.
\end{proof}

A second interesting  consequence of Lemma \ref{Lem:speed-and-model} is:

\begin{corollary}\label{Cor:indep-paso-iper}
Let $(\phi_t)$ be a semigroup in $\D$.  If there exist $z\in \D$ and $u> 0$ such that the hyperbolic step of order $u$ is equal to $0$, that is $s_u(\phi_t,z)=0$, then for every $v\geq 0$ and $w\in \D$ it holds $s_v(\phi_t,w)=0$.
\end{corollary}
\begin{proof}
Let $(\Omega, h, \psi_t)$ be a holomorphic model for $(\phi_t)$ given by Theorem~\ref{Thm:existmodel}, with either $\Omega=\D$ or $\Omega=\C$.

Let $v\geq 0$, and $z\in \D$. Set $w:=\phi_v(z)$. By Lemma \ref{Lem:speed-and-model}
\begin{equation}\label{Eq:iper-step-dis}
\begin{split}
k_\Omega(h(z), h(w))&=\lim_{t\to +\infty}\omega(\phi_t(z), \phi_t(w))\\&=\lim_{t\to +\infty}\omega(\phi_t(z), \phi_{t+v}(z))=s_v(\phi_t,z).
\end{split}
\end{equation}
Assume $s_u(\phi_t,z)=0$ for some $u>0$ and $z\in \D$, and let $w:=\phi_u(z)$.
Note that $w\neq z$, since $(\phi_t)$ is not elliptic. Moreover, $h$ is injective, hence $h(z)\neq h(w)$. Therefore, by \eqref{Eq:iper-step-dis}, the domain $\Omega$ has two different points whose hyperbolic distance is zero. Hence $\Omega=\C$ and $k_\Omega\equiv 0$. Then by \eqref{Eq:iper-step-dis}
 the result holds.
\end{proof}

\begin{definition}
Let $(\phi_t)$ be a semigroup in $\D$. We say that $(\phi_t)$ is of {\sl positive hyperbolic step} (or it is of {\sl automorphic type}) \index{Positive hyperbolic step} \index{Semigroup of automorphic type} \index{Semigroup!of positive hyperbolic type} if there exist $z\in \D$ and $u\geq 0$ such that $s_u(\phi_t, z)>0$. Otherwise, we say that $(\phi_t)$ is of {\sl zero hyperbolic step} \index{Zero hyperbolic step} (or it is of {\sl non-automorphism type}). \index{Semigroup of non-automorphism type}\index{Semigroup!of zero hyperbolic type}
\end{definition}

By Remark \ref{Rem:step-elliptic}, groups of automorphisms of $\D$ are of positive hyperbolic step, while elliptic semigroups in $\D$ which are not groups, are always of zero hyperbolic step. By Remark \ref{Rem:step-elliptic} and Corollary \ref{Cor:indep-paso-iper} it also holds that if  $(\phi_t)$ is non-elliptic, then it is of positive hyperbolic step if and only if $s_u(\phi_t, z)> 0$ for some---and hence for all---$u> 0$ and $z\in \D$.

Let us denote by  $\Ha:=\{\zeta \in \C: \Re \zeta>0\}$, $\Ha^{-}:=\{\zeta\in \C: \Re \zeta<0\}$ and, given $\rho>0$, $\strip_\rho:=\{\zeta\in \C: 0<\Re \zeta<\rho\}$.

\begin{theorem}\label{Thm:modelholo}
Let $(\phi_t)$ be a  semigroup in $\D$. Then:
\begin{enumerate}
\item  $(\phi_t)$ is hyperbolic with spectral value $\lambda>0$ if and only if  it has a holomorphic model $(\strip_{\frac{\pi}{ \lambda}}, h, z\mapsto z+it)$.
\item  $(\phi_t)$ is parabolic of positive hyperbolic step if and only if it has a holomorphic model either of the form $(\Ha, h, z\mapsto z+it)$ or of the form $(\Ha^-, h, z\mapsto z+it)$.
\item $(\phi_t)$ is parabolic of zero hyperbolic step if and only if
it has a holomorphic model $(\C, h, z\mapsto z+it)$.
\end{enumerate}
\end{theorem}
\begin{proof}
(1) Suppose $(\phi_t)$ is hyperbolic, with Denjoy-Wolff point $\tau\in \partial \D$ and spectral value $\lambda>0$. By Theorem \ref{Thm:existmodel}, there exists a holomorphic model $(\Omega, h, \psi_t)$ for $(\phi_t)$ with $\Omega=\C$ or $\Omega=\D$. By \eqref{Eq:speed-divergence} and Lemma \ref{Lem:speed-and-model},  it holds
\[
0<\frac{1}{2}\lambda=c_\D(\phi_t)=c_\Omega(\psi_t).
\]
In particular, the hyperbolic distance of $\Omega$ is not identically zero, and therefore $\Omega=\D$. Moreover, since $c_\D(\psi_t)>0$ and, again by \eqref{Eq:speed-divergence}, $(\psi_t)$ is a hyperbolic group of automorphisms of $\D$ with the same spectral value $\lambda$. Using a M\"obius transformation $C$ such that $C(\D)=\Ha$, we can find an isomorphic holomorphic model for $(\phi_t)$ given by $(\Ha, C\circ h,  z\mapsto e^{\lambda t}z)$. Let $\log z$ denote the principal branch of the logarithm on $\Ha$, and let $f:\Ha \to \strip_{\frac{\pi}{\lambda}}$ be the biholomorphism given by
\begin{equation}\label{Eq:iper-converte-para}
f(z):=\frac{i}{\lambda}\log z + \frac{\pi}{2\lambda}.
\end{equation}
Note that $f(e^{\lambda t}z)=f(z)+it$ for all $z\in \Ha$ and $t\geq 0$. Therefore it is easy to see that $(\strip_{\frac{\pi}{\lambda}}, f\circ C\circ h, z\mapsto z+it)$ is a holomorphic model for $(\phi_t)$.

On the other hand, if $(\strip_{\frac{\pi}{\lambda}}, h, z\mapsto z+it)$ is a holomorphic model for $(\phi_t)$, let $f$ be the map defined in \eqref{Eq:iper-converte-para}. Then $(\Ha, f^{-1} \circ h, z\mapsto e^{\lambda t }z)$ is a holomorphic model for $(\phi_t)$. The group $(z\mapsto e^{\lambda t }z)$ is conjugated via a M\"obius transformation $C$ which maps $\D$ onto $\Ha$ to a hyperbolic group $(\psi_t)$ of $\D$ with spectral value $\lambda$. Therefore, $(\D, C^{-1} \circ f^{-1} \circ h, \psi_t)$ is a holomorphic model for $(\phi_t)$. By \eqref{Eq:speed-divergence} and Lemma \ref{Lem:speed-and-model}, it follows that $c_\D(\phi_t)=\frac{\lambda}{2}>0$ and hence, since $(\phi_t)$ can not be elliptic by Remark \ref{Rem:speed-elliptic}, $(\phi_t)$ is a hyperbolic semigroup in $\D$ with  spectral value $\lambda$.

(2)  Assume $(\phi_t)$ is parabolic with positive hyperbolic step. By Theorem \ref{Thm:existmodel}, there exists a holomorphic model $(\Omega, h, \psi_t)$, with $\Omega=\D$ or $\C$. By \eqref{Eq:distance-limit},
\[
k_\Omega(h(z), \psi_1(h(z)))=k_\Omega(h(z), h(\phi_1(z)))=\lim_{t\to \infty}\omega(\phi_t(z), \phi_{t+1}(z))=s_1(\phi_t,z)>0,
\]
 hence $\Omega=\D$. The group $(\psi_t)$ has divergence rate $c_\D(\psi_t)=c_\D(\phi_t)=0$ by Lemma \ref{Lem:speed-and-model}. Since $(\phi_t)$ has no common fixed point, $(\psi_t)$ is a group of parabolic automorphisms of $\D$. Conjugating with a suitable M\"obius transformation, it follows that a holomorphic model for $(\phi_t)$ is either $(\Ha, h, z\mapsto z+it)$ or $(\Ha^-, h, z\mapsto z+it)$. Conversely, if $(\phi_t)$ admits a holomorphic model of the forms $(\Ha, h, z\mapsto z+it)$ or $(\Ha^-, h, z\mapsto z+it)$, then by conjugating with a M\"obius transformation, it follows that $(\phi_t)$ has a holomorphic model $(\D, h, \psi_t)$ with $(\psi_t)$ a group of parabolic automorphisms of $\D$. By Lemma \ref{Lem:speed-and-model}, it follows then that $(\phi_t)$ is parabolic with positive hyperbolic step.

(3) The proof follows from an argument similar to (2), just noting that by Lemma \ref{Lem:speed-and-model}, $(\phi_t)$ is of zero hyperbolic step if and only if the domain $\Omega$ given in Theorem  \ref{Thm:existmodel} is $\C$. Moreover, $(\psi_t)$ can not have a fixed point in $\C$, for otherwise $(\phi_t)$ would be elliptic. Therefore $(\psi_t)$ is a group of translations in $\C$.
\end{proof}

The previous theorem has the following well-known consequence.

\begin{corollary}
Let $(\phi_t)$ be a hyperbolic semigroup in $\D$. Then $(\phi_t)$ is of positive hyperbolic step.
\end{corollary}
\begin{proof}
By Theorem \ref{Thm:modelholo}, $(\phi_t)$ has a holomorphic model of the form $(\strip_\rho, h, z\mapsto z+it)$ for some $\rho>0$. In particular, by \eqref{Eq:distance-limit},
\[
s_1(\phi_t,0)=\lim_{t\to \infty}\omega(\phi_t(0), \phi_{t+1}(0))=k_{\strip_{\rho}}(h(0), h(\phi_1(0)))=k_{\strip_{\rho}}(h(0), h(0)+i)>0,
\]
hence $(\phi_t)$ is of positive hyperbolic step.
\end{proof}

\section{From K\"onigs functions to infinitesimal generators}
Theorem \ref{Thm:modelholo} provides simple holomorphic models for a non-elliptic semigroup $(\phi_t)$ in $\D$, which also give information on hyperbolic and dynamical properties of $(\phi_t)$. Indeed, the intertwining map $h$ (and its image) contains  all the information about the semigroup $(\phi_t)$. These models and maps deserve a special name:

\begin{definition} \label{Def:Koenigs_function}
Let $(\phi_t)$ be a semigroup in $\D$. The holomorphic model for $(\phi_t)$ given by Theorem \ref{Thm:modelholo} is called the {\sl canonical model} \index{Canonical model} of $(\phi_t)$ and the  intertwining map $h$ is the {\sl K\"{o}nigs function} \index{K\"{o}nigs function} associated with $(\phi_t)$.
\end{definition}

At a first sight, the use of the definite article ``the'' for denoting a canonical model might not seem to be a good idea. Indeed, although Theorem \ref{Thm:modelholo} assures that every canonical model for a given non-elliptic semigroup of $\D$ has the same  base space and  group of automorphisms,  a priori, there could be other ``K\"onigs functions''  intertwining the given semigroup with the same group of automorphisms. This is the case, in fact, but it turns out that all K\"onigs functions are essentially unique up to a constant, so that the use of the definite article ``the'' is well justified:

\begin{proposition}\label{Prop:unique-Koenigs}\index{K\"{o}nigs function!Uniqueness}
Let $(\phi_t)$ be a non-elliptic semigroup in $\D$. Let $(\Omega, h, \psi_t)$ be a canonical model for $(\phi_t)$ and let $\tilde{h}:\D \to \Omega$ be holomorphic.  Then
\begin{enumerate}
\item If $(\phi_t)$ is either  hyperbolic or parabolic of positive hyperbolic step, then  $\tilde{h}$ is a K\"onigs function for $(\phi_t)$ if and only if there exists $a\in \R$ such that $\tilde{h}(z)=h(z)+ai$ for all $z\in \D$.
\item If $(\phi_t)$ is  parabolic of zero hyperbolic step, then $\tilde{h}$ is a K\"onigs function for $(\phi_t)$ if and only if there exists $a\in \C$ such that $\tilde{h}(z)=h(z)+a$ for all $z\in \D$.
\end{enumerate}
\end{proposition}

\begin{proof}
The ``if'' implications of the statements are clear. Thus, assume $\tilde{h}$ is a K\"onigs function for $(\phi_t)$. Then, there exists an automorphism $\nu:\Omega\to \Omega$ such that $\tilde{h}=\nu \circ h$ and for all $t\geq 0$ it holds $\nu \circ \psi_t=\psi_t \circ \nu$.
Since $(\phi_t)$ is non-elliptic, for all $z\in \Omega$ and all $t\geq 0$ it holds
\begin{equation}\label{Eq:non-elliptic-funct}
\nu(z+it)=\nu(z)+it.
\end{equation}
Differentiating \eqref{Eq:non-elliptic-funct} in $t$ and setting $t=0$ we obtain $\nu'(z)\equiv 1$. Integrating in $z$, we obtain $\nu(z)=z+c$ for some $c\in \C$.
Moreover, if $(\phi_t)$ is either hyperbolic or parabolic of positive hyperbolic step, since $\Omega$ is a half plane or a strip, and $\nu(\Omega)=\Omega$, then $c$  is pure imaginary. From this, (1) and (2) hold.
\end{proof}

K\"onigs functions associated with non-elliptic semigroups are related with a very specific type of univalent functions. We start with the following simple result.

\begin{proposition}\label{Prop:h-limit-tau}
Let $(\phi_t)$ be a non-elliptic semigroup in $\D$ with Denjoy-Wolff point $\tau\in \de\D$. Let $h$ be the associated K\"onigs function. Then
\begin{equation}\label{Eq:lim-Im-h-at-DW}
\limsup_{z\to \tau} \Im h(z)=+\infty.
\end{equation}
\end{proposition}
\begin{proof}
Let  $\gamma:[0,+\infty)\to \D$ be the continuous curve defined by $\gamma(t):=\phi_t(0)$. Note that $\lim_{t\to+\infty}\gamma(t)=\tau$ by the continuous Denjoy-Wolff theorem. Hence,
\[
\lim_{t\to+\infty}  \Im h(\gamma(t))=\lim_{t\to+\infty}  \Im h(\phi_t(0)))=\lim_{t\to+\infty}  \Im (h(0)+it)=+\infty,
\]
and the result holds. \end{proof}

\begin{remark}\label{Rem:behavior h not good DW}
Arguing as in the proof of the previous proposition and using the Lehto-Virtanen's Theorem, it follows that under the hypotheses of Proposition~\ref{Prop:h-limit-tau},
\[
\angle\lim_{z\to \tau}h(z)=\infty,
\]
where the limit has to be understood in $\C_\infty$. However, such a condition by itself does not characterize the Denjoy-Wolff point of $(\phi_t)$, while \eqref{Eq:lim-Im-h-at-DW} does, as we will see later on.
\end{remark}

\begin{definition}
A domain $\Omega\subset \C$  is {\sl starlike at infinity} if  $(\Omega+it)\subseteq \Omega$ for all $t\geq 0$.

A map $h:\D \to \C$ is a {\sl starlike at infinity with respect  to $\sigma\in \partial \D$} \index{Starlike at infinity} if it is univalent,  $\limsup_{z\to \sigma}\Im h(z)=+\infty$ in $\C_\infty$ and $h(\D)$ is starlike at infinity.
\end{definition}

Maps which are starlike at infinity and K\"{o}nigs functions of non-elliptic semigroups in $\D$ are one and the same:

\begin{theorem}\label{Thm:estrella-infinito-koenigs}
Let $h:\D \to \C$ be the K\"{o}nigs function of a non-elliptic semigroup $(\phi_t)$ in $\D$ with Denjoy-Wolff point $\tau\in \partial \D$. Then $h$ is starlike at infinity with respect to $\tau$.

Conversely, if $h:\D \to \C$ is starlike at infinity with respect to $\tau\in \partial \D$, let $\phi_t(z):=h^{-1}(h(z)+it)$ for $t\geq 0$. Then $(\phi_t)$ is a non-elliptic semigroup  in $\D$ with Denjoy-Wolff point $\tau$. Moreover, let $a=\inf_{z\in \D} \Re h(z)$ and $b=\sup_{z\in \D} \Re h(z)$, then
\begin{enumerate}
\item if $a=-\infty, b=+\infty$, then $(\phi_t)$ is a parabolic semigroup of zero hyperbolic step and $h$ is its K\"onigs function,
\item if  $a=-\infty$ and $b<+\infty$, then $(\phi_t)$ is a parabolic semigroup of positive hyperbolic step, and its canonical model is $(\Ha^-, h-b, z\mapsto z+it)$,
\item if $a>-\infty$ and $b=+\infty$, then $(\phi_t)$ is a parabolic semigroup of positive hyperbolic step, and its canonical model is $(\Ha, h-a, z\mapsto z+it)$,
\item if $a>-\infty$ and $b<+\infty$ then $h$ is a hyperbolic semigroup with spectral value $\lambda=\frac{\pi}{b-a}$ and $h-a$ is its K\"onigs function.
\end{enumerate}
\end{theorem}

\begin{proof}
Let $h$ be the K\"{o}nigs function of a non-elliptic semigroup $(\phi_t)$ in $\D$ with Denjoy-Wolff point $\tau\in \partial \D$. By  Theorem \ref{Thm:modelholo}, for all $t>0$
\[
h(\D)+it=h(\phi_t(\D))\subset h(\D),
\]
hence $h(\D)$ is starlike at infinity. Moreover, $\limsup_{z\to \tau} \Im h(z)=+\infty$ by Proposition~\ref{Prop:h-limit-tau}. Hence $h$ is starlike at infinity with respect to $\tau$.

Conversely, let us assume $h$ is starlike at infinity with respect to $\tau$. Define $\phi_t(z):=h^{-1}(h(z)+it)$, for all $z\in \D$ and $t\geq 0$. Now, it is clear that $(\phi_t)$ is a (continuous) semigroup in $\D$, without fixed points in $\D$.

First, note that if $z_0\in h(\D)$ then $\cup_{t\geq 0}(z_0-it)=\{z\in \C: \Re z=z_0, \Im z\in (-\infty, \Im z_0]\}$. Since moreover $h(\D)+it\subset h(\D)$ for all $t\geq 0$, it follows that if $z_0\in h(\D)$ then $\cup_{t\geq 0}(h(\D)-ti)$ contains the line $\{z\in \C: \Re z=\Re z_0\}$.  Therefore,  bearing in mind that $h(\D)$ is connected,
\[
\Omega:=\cup_{t\geq 0}(h(\D)-ti)=\{z\in \C: a<\Re z <b\}.
\]
From this and from Theorem \ref{Thm:modelholo}, taking into account that  $h+\alpha$ satisfies \eqref{Eq:unosemigroup} for all $\alpha\in \R$,  implications (1), (2), (3) and (4) follow easily. As an example, let us prove (2). In this case, $\Omega=\{z\in \C: \Re z<b\}$. Therefore, if we let $\tilde{h}:=h-b$, it follows that $\tilde{h}$ satisfies \eqref{Eq:unosemigroup}, and since $\cup_{t\geq 0}(\tilde{h}(\D)-it)=\Ha^-$,  \eqref{Eq:duesemigroup} is satisfied as well. Therefore, $(\Ha^-, h-b, z\mapsto z+it)$ is the canonical model of $(\phi_t)$, and $(\phi_t)$ is a parabolic semigroup with positive hyperbolic step by the Theorem \ref{Thm:modelholo}.

Finally, we are left to show that $\tau$ is the Denjoy-Wolff point of $(\phi_t)$.

Assume this is not the case, and let $\sigma\in \partial \D$ be the Denjoy-Wolff point of $(\phi_t)$. Let $\tilde{h}$ be the K\"onigs function of $(\phi_t)$. By Proposition~\ref{Prop:h-limit-tau}, there exists a sequence $\{w_n\}\subset \D$ converging to $\sigma$ such that $\lim_{n\to \infty}\Im \tilde{h}(w_n)=+\infty$. As we already noticed above, $h=\tilde{h}-\alpha$ for some $\alpha\in \R$, therefore, $\lim_{n\to \infty}\Im \tilde{h}(w_n)=+\infty$ as well.

Since $h$ is starlike at infinity with respect to $\tau$,  there exists a sequence $\{z_n\}\subset \D$ converging to $\tau$ such that $\lim_{n\to \infty}\Im h(z_n)=+\infty$.

The connected domain $h(\D)$ being  starlike at infinity,  there exists a curve $\Gamma_n\subset h(\D)$ joining $z_n$ to $w_n$ such that
\[
\min_{\zeta\in \Gamma_n} \Im \zeta=\min\{\Im h(z_n), \Im h(w_n)\}.
\]
In particular, for every $R>0$ there exists $n_R\in \N$ such that for all $n\geq n_R$ and all $\zeta \in \Gamma_n$, it holds $\Im \zeta\geq R$. This implies that if $\{\zeta_n\}$ is a sequence such that $\zeta_n\in \Gamma_n$ for all $n\in \N$, then
\begin{equation}\label{Eq:wn-goes-infty-Konigs}
\lim_{n\to \infty} \zeta_n=\infty \quad \hbox{\ in \ } \C_\infty.
\end{equation}

Now, let $C_n:=h^{-1}(\Gamma_n)$. By construction, $C_n$ joins $z_n$ to $w_n$, therefore there exists $K>0$ such that $\hbox{diam}_{\mathcal E}(C_n)\geq K$ for all $n\in \N$. Moreover, for any sequence $\{\xi_n\}$ such that $\xi_n\in C_n$, it holds $h(\xi_n)\in \Gamma_n$. Hence,  by \eqref{Eq:wn-goes-infty-Konigs}, the sequence $\{h(\xi_n)\}$ converges to $\infty$ in $\C_\infty$.

Therefore, $\{C_n\}$ is a sequence of Koebe arcs for $h$, contradicting the no Koebe arcs Theorem. Hence $\tau=\sigma$ and we are done.
\end{proof}

Next theorem was first proved under slightly different hypotheses in \cite{Gru71} for holomorphic functions from the upper half-plane and in \cite{Cio78} for the unit disc case.

\begin{theorem}\label{Thm:starlike-infinity}
Let $h:\D \to \C$ be  non-constant and holomorphic  and let $\sigma\in \partial \D$. Then $h$ is starlike at infinity \index{Starlike at infinity} with respect to $\sigma$ if and only if for all $z\in \D$ it holds
\begin{equation}\label{Eq:starlike-at-infinity}
\Im [\overline{\sigma}(\sigma-z)^2h'(z)]\geq 0
\end{equation}
Moreover, equality holds at some --- and hence any --- $z\in \D$ if and only if
\begin{equation}\label{Eq:starlike-inf-igual}
h(z)=a\frac{\sigma+z}{\sigma-z}+c, \quad z\in \D
\end{equation}
for some $a\in \R\setminus\{0\}$ and $c\in \C$.
\end{theorem}

\begin{proof}
Let $C_\sigma:\D \to \Ha$ be the Cayley transform with respect to $\sigma$ given by $\sigma(z)=\frac{\sigma+z}{\sigma-z}$.  Given $h:\D \to \C$ holomorphic, we define a new holomorphic map $g:\Ha \to \C$ by setting $g(w):=h(C_\sigma^{-1}(w))$. Thus, writing $z=C_\sigma^{-1}(w)$,
\[
g'(w)=h'(C_\sigma^{-1}(w))\cdot(C_\sigma^{-1}(w))'=\frac{h'(z)}{C_\sigma'(z)}=\frac{\overline{\sigma}}{2}(\sigma-z)^2h'(z).
\]
Therefore, \eqref{Eq:starlike-at-infinity} is equivalent to
\begin{equation}\label{Eq:starlike-at-infinity-semiplan}
\Im g'(w)\geq 0, \quad \hbox{for all \ }w\in \Ha.
\end{equation}

Moreover, it is easy to see that $h$ is starlike at infinity with respect to $\sigma\in \partial \D$ if and only if $g$ is {\sl starlike at infinity}, namely, $g$ is univalent, for every $t\geq 0$ it holds  $g(\Ha)+it\subset g(\Ha)$ and $\limsup_{w\to \infty}\Im g(w)=+\infty$.

Thus, in order to prove the result, we have to prove that $g$ is starlike at infinity if and only if $g$ is not constant and \eqref{Eq:starlike-at-infinity-semiplan} holds.

To start with, assume that \eqref{Eq:starlike-at-infinity-semiplan} holds. First of all, note that if there exists $w_0\in \Ha$ such that $\Im g'(w_0)=0$, by the Maximum Principle for harmonic functions, it holds $\Im g'(w)\equiv 0$, and hence $g'(w)\equiv a$ for some $a\in \R$. Namely, $g(w)=aw+b$ for some $b\in \C$. Since $g$ is not constant, $a\neq 0$ and it is easy to see that $g$ is starlike at infinity. Moreover, a direct computation shows that $h$ has the form \eqref{Eq:starlike-inf-igual}.

Assume now that  $\Im g'(w)>0$ for all $w\in \Ha$.  By Noshiro's Theorem, the function $g$ is univalent. Also, let $y>0$ and $r\in \R$. Since
\[
\frac{\de}{\de r} \Re g(y+ir)=\Re(i g'(y+ir))=-\Im g'(y+ir)<0,
\]
it follows that the function $\R\ni r\mapsto \Re g(y+ir)$ is decreasing.  This implies that the curve $\R\ni r\mapsto g(y+ir)$, that parametrizes $\partial g(E^\Ha(\infty,y))$, intersects every vertical line at most in one point --- here  $E^\Ha(\infty,y)$ denotes the horocycle of $\Ha$ at $\infty$  of radius $y$.
Therefore, either $g(y+ir)+it\in g(E^\Ha(\infty, y))$ for all $t<0$ or $g(y+ir)+it\in g(E^\Ha(\infty, y))$ for all $t>0$. However, the first possibility is excluded because $g$ preserves the orientation. Therefore, for every $t\geq 0$, and for every $y>0$ it holds
\begin{equation}\label{Eq:g-go-inside}
g(E^\Ha(\infty, y))+it\subset g(E^\Ha(\infty,y)).
\end{equation}

By the arbitrariness of $y$, it follows that $g(\Ha)+it\subset g(\Ha)$ for all $t\geq 0$. Therefore, in order to prove that $g$ is starlike at infinity, it remains only to show that $\limsup_{w\to \infty}\Im g(w)=+\infty$. Suppose by contradiction that $\limsup_{w\to \infty}\Im g(w)=A<+\infty$. We claim that this implies that $g(\Ha)$ is contained in $\{w\in \C: \Im w<A\}$, which is clearly impossible since $g(\Ha)$ is starlike at infinity. Suppose the claim is false. Hence there exists $w_0=x_0+iy_0\in \Ha$ such that $\Im g(w)>A+\epsilon$, for some $\epsilon>0$. Since $\Im g'(w)>0$ for all $w\in \Ha$,  the curve $[x_0,+\infty) \ni r\mapsto \Im g(r+iy_0)$ is increasing, hence $\limsup_{r\to +\infty}  \Im g(r+iy_0)\geq A+\epsilon$, a contradiction. This proves that $g$ is starlike at infinity, as needed.

Assume now that $g$ is starlike at infinity. We want to show that \eqref{Eq:starlike-at-infinity-semiplan} holds.
First, we claim that

\medskip
{$(\ast)$} $g$  starlike at infinity implies \eqref{Eq:g-go-inside}.
\medskip

Assume  the claim {$(\ast)$} is true. Fix $y>0$. Then the function $\R\ni r\mapsto \Re g(y+ir)$ is either monotone or constant. Indeed, assume it is not constant, and that, by contradiction, there exist $r_0, r_1\in \R$ such that $\Re g(y+ir_0)=\Re g(y+ir_1)$. Since $g$ is univalent, $\Im g(y+ir_0)\neq \Im g(y+ir_1)$ and we can assume that $\Im g(y+ir_0)> \Im g(y+ir_1)$. Let $t:=\Im g(y+ir_0)-\Im g(y+ir_1)>0$. By \eqref{Eq:g-go-inside},
\[
g(y+ir_0)=g(y+ir_1)+it\in g(E^\Ha(\infty,y)),
\]
which is  a contradiction. Therefore, if not constant, the function $\R\ni r\mapsto \Re g(y+ir)$ is monotone. Moreover,  by \eqref{Eq:g-go-inside}, it is clear that $g$ maps $E^\Ha(\infty,y)$ onto the connected component of $\C\setminus g(\partial E^\Ha(\infty,y))$ which  contains the curve $(0,+\infty)\ni r\mapsto g(y+ir)$.  Since univalent maps  preserve orientation, this implies that, if not constant, the function $\R\ni r\mapsto \Re g(y+ir)$ is strictly decreasing. In particular, $\Im g'(y+ir)=-\frac{\de}{\de r}\Re g(y+ir)\geq 0$,  and \eqref{Eq:starlike-at-infinity-semiplan} holds.

Now, we  show that claim $(\ast)$ holds. Fix $y>0$. By Theorem \ref{Thm:estrella-infinito-koenigs}, there exists $\alpha\in \R$ such that $h+\alpha$ is the K\"onigs function of a non-elliptic semigroup $(\phi_t)$ in $\D$ with Denjoy-Wolff point $\sigma$.  In particular, for all $t>0$, taking into account that $\phi_t(z)=h^{-1}(h(z)+it))$ and the continuous Denjoy-Wolff  theorem, it holds
\[
h^{-1}\left(h(E(\sigma,\frac{1}{y})) +it\right)\subseteq E(\sigma,\frac{1}{y}).
\]
Hence
\begin{equation*}
\begin{split}
g(E^\Ha(\infty, y))+it&=h(C_\sigma^{-1}(E^\Ha(\infty, y)))+it=
h(E(\sigma,\frac{1}{y})) +it\\& \subseteq h(E(\sigma,\frac{1}{y}))=g(C_\sigma(E(\sigma,\frac{1}{y})))=g(E^\Ha(\infty, y)),
\end{split}
\end{equation*}
and claim $(\ast)$ is proved.

Finally,  it is clear that if $h$ is given by \eqref{Eq:starlike-inf-igual}, then $\Im [\overline\sigma (\sigma-z)^2 h'(z) ]=0$ for all $z\in \D$.
\end{proof}

From our previous construction and analysis of the K\"onigs function, we finally show how to generate the infinitesimal generator of a non-elliptic semigroup as well as how to deduce the Berkson-Porta decomposition theorem.

\begin{theorem} Let $(\phi_t)$ be a non-elliptic semigroup in $\D$ with Denjoy-Woll point $\tau\in\partial\D$. Then the function
\[
[0,+\infty)\times \D\ni (t,z) \longmapsto \phi_t(z)\in \D
\]
is real analytic and there exists a unique holomorphic function $G:\D\to \C$ such that
\begin{equation}\label{Eq:Existence inf generator}
\frac{\de \phi_t(z)}{\de
t}=G(\phi_t(z))\qquad z\in \D, \ t\in [0,+\infty).
\end{equation}
Moreover, let $h$ be the K\"onigs function of $(\phi_t)$.
\begin{enumerate}
\item For all $z\in \D$, it holds
\[
G(z)=\frac{i}{h'(z)}.
\]

\item If
$$
p(z):=\frac{i}{h'(z)(z-\tau)(\overline \tau z-1)}, \qquad z\in \D,
$$
then $G(z)=(z-\tau)(\overline{\tau}z-1)p(z)$ and $\Re p(z)\geq0$, for $z\in\D$.
\end{enumerate}
\end{theorem}

\begin{proof} According to Theorem \ref{Thm:modelholo}, the canonical model of $(\phi_t)$ is $(\Omega, h, z+it)$, with either $\Omega=\strip_\rho$ for some $\rho>0$ or $\Omega=\Ha$ or $\Omega=\Ha^-$ or $\Omega=\C$. In particular, this means that $\phi_t(z)=h^{-1}(h(z)+it)$. This implies that the map $(t,z)\mapsto \phi_t(z)$ is real analytic in $[0,+\infty)\times \D$. Let $G(z):=\left.\frac{\partial \phi_t(z)}{\partial t} \right|_{t=0}=\frac{i}{h'(z)}$, $z\in \D$. This function is holomorphic and

$$
\frac{\de \phi_t(z)}{\de t}=\frac{i}{h'(\phi_t(z))}=G(\phi_t(z)), \quad z\in \D, \ t\geq 0.
$$
That is,  $(\phi_t)$ solves \eqref{Eq:Existence inf generator}. Moreover, the function $h$ is  starlike at infinity with respect to $\tau$ (see Theorem \ref{Thm:estrella-infinito-koenigs}). Therefore, by Theorem \ref{Thm:starlike-infinity},
$$
\Re p(z)=\Re \overline{p(z)}=\frac{\Im (h'(z)(z-\tau)(\overline \tau z-1))}{|h'(z)(z-\tau)(\overline \tau z-1)|^2} \geq 0, \qquad z\in \D.
$$
Hence $G(z)=(z-\tau)(\overline{\tau}z-1)p(z)$.
\end{proof}

\end{document}